\documentclass
{article}
\usepackage[latin1]{inputenc}
\usepackage{indentfirst}

\usepackage{amsmath,amsfonts,amssymb,amsthm,amscd}
\RequirePackage{ifthen}
\RequirePackage{hyperref}

\usepackage{latexsym}
\usepackage{mathrsfs}
\usepackage{algorithm}
\usepackage[noend]{algorithmic}
\usepackage{paralist}
\usepackage{graphicx}
\usepackage{booktabs}
\usepackage{mathtools}
\usepackage{rotating}
\usepackage{tikz}

\usepackage{subcaption}

\usepackage[normalem]{ulem}
\usepackage[color,matrix,arrow]{xy}
\usepackage{mathtools}

\usetikzlibrary{patterns}

\usepackage{comment}

\usepackage[a4paper,margin=1in]{geometry}

\input xy
\xyoption{all}
\newcommand{\leqdr}{\mathbin{\rotatebox[origin=c]{-45}{$\leq$}}}
\newcommand{\lequr}{\mathbin{\rotatebox[origin=c]{45}{$\leq$}}}
\newdir{d}{\leq}
\newdir{d}{\geq}
\newdir{dr}{\leqdr}
\newdir{ur}{\lequr}

\theoremstyle{plain}

\theoremstyle
{plain}
\newtheorem{theorem}{Theorem}[section]

\newtheorem{proposition}[theorem]{Proposition}

\newtheorem{lemma}[theorem]{Lemma}
\newtheorem{corollary}[theorem]{Corollary}

\newtheorem{question}[theorem]{Question}
\newtheorem{claim}[theorem]{Claim}
\theoremstyle{definition}
\newtheorem{definition}[theorem]{Definition}

\newtheorem{example}[theorem]{Example}
\newtheorem{remark}[theorem]{Remark}

\newtheorem{theoremintro}{Theorem}

\makeindex

\newcommand{\N}{\mathbb{N}}

\newcommand{\Q}{\mathbb{Q}}
\newcommand{\R}{\mathbb{R}}

\newcommand{\Addresses}{{
		\bigskip
		\footnotesize
		
		\noindent N.~Zava,  \textsc{Institute of Science and Technology Austria (ISTA), 3400 Klosterneuburg, Austria}\\
		\textit{E-mail address}: \texttt{nicolo.zava@gmail.com}
}}

\newcommand{\EH}{d_{EH}}
\newcommand{\GH}{d_{GH}}

\newcommand{\GHs}{\mathcal{GH}}
\newcommand{\EHs}{\mathcal{EH}}

\DeclareMathOperator{\diam}{diam}

\DeclareMathOperator{\asdim}{asdim}

\DeclareMathOperator{\Iso}{Isom}

\DeclareMathOperator{\dist}{dist}
\DeclareMathOperator{\dis}{dis}

\author{ 
Nicol\`o Zava
}
\title{Coarse and bi-Lipschitz embeddability of subspaces of the Gromov-Hausdorff space into Hilbert spaces}
\date{}

\begin{document}
	\maketitle

\begin{abstract}
In this paper, we discuss the embeddability of subspaces of the Gromov-Hausdorff space, which consists of isometry classes of compact metric spaces endowed with the Gromov-Hausdorff distance, into Hilbert spaces. These embeddings are particularly valuable for applications to topological data analysis.
We prove that its subspace consisting of metric spaces with at most n points has asymptotic dimension $n(n-1)/2$. Thus, there exists a coarse embedding of that space into a Hilbert space. 
On the contrary, if the number of points is not bounded, then the subspace cannot be coarsely embedded into any uniformly convex Banach space and so, in particular, into any Hilbert space.
Furthermore, we prove that, even if we restrict to finite metric spaces whose diameter is bounded by some constant, the subspace still cannot be bi-Lipschitz embedded into any finite-dimensional Hilbert space. 
We obtain both non-embeddability results by finding 
obstructions to coarse and bi-Lipschitz embeddings in families of isometry classes of finite subsets of the real line endowed with the Euclidean-Hausdorff distance.
\end{abstract}
	

{\footnotesize
\noindent MSC2020: 51F30, 
46B85, 
54B20. 
}

 \footnotesize
{
\noindent Keywords: {Gromov-Hausdorff distance, Euclidean-Hausdorff distance, stable invariants, asymptotic dimension, coarse embeddings, Assouad dimension, bi-Lipschitz embeddings.}

\normalsize 
	
\section{Introduction}

The  Gromov-Hausdorff distance $\GH$ measures how two metric spaces resemble each other. It was introduced by Edwards in  \cite{Edw}, and then rediscovered and generalised by Gromov (\cite{Gro_GH_fr
	})
	. Until around 2000
	, the Gromov-Hausdorff distance had been mainly used by pure mathematicians who were interested in the induced topology. That direction is still of great interest, and, as an example, we mention the two recent papers \cite{Ant1,Ant2}.
	
	In addition to the intrinsic interest in it, a 
 great impulse to study the quantitative aspects of the Gromov-Hausdorff distance came from its applications in topological data analysis, which is a fast-growing subject aiming to use topological techniques to analyse a wide range of real-world data 
 (see, for example, \cite{Car,Hes}, and to \cite{GiuLaz} for a growing dataset of real-world applications).   The Gromov-Hausdorff distance provides a theoretical framework to directly compare point clouds by considering them as metric spaces. This approach  
 proved to be 
		useful 
		in shape recognition and comparison (\cite{MemSap1,MemSap2,Mem07}), which arises
		, for example, in 
		molecular biology, databases of objects, face recognition and matching of articulated objects. 

Comparing two metric spaces using the Gromov-Hausdorff distance directly is computationally expensive. 
Even approximating it within a factor of $3$ for trees with unit edge length is NP-hard (\cite{AgaFoxNatSidWan,Sch}, see also \cite{Mem07}, where the author discussed the connection between computing the Gromov-Hausdorff distance and a class of NP-hard problems). Therefore, creating efficiently computable invariants to approximate the Gromov-Hausdorff distance is of particular interest. Following \cite{Mem12}, an {\em invariant} $\psi$ associates to a metric space $X$ an element $\psi(X)$ of another metric space $(\mathfrak Y,d_{\mathfrak Y})$ in such a way that, if $X$ and $Y$ are two isometric metric spaces, then $\psi(X)=\psi(Y)$. Furthermore, an invariant $\psi$ is {\em stable} if  there exists a function $\rho_+\colon\R_{\geq 0}\to\R_{\geq 0}$  such that 
\begin{equation}\label{eq:borno}d_{\mathfrak Y}(\psi(X),\psi(Y))\leq\rho_+(\GH(X,Y))\end{equation}
for every pair of metric spaces $X$ and $Y$. Stability implies that small perturbations of the metric spaces have a limited effect on the associated invariants. Therefore, considering similarity recognition, we avoid false negatives
, which are situations where two metric spaces are very similar in the Gromov-Hausdorff distance, but their invariants are far apart. Furthermore, stable invariants can be used to provide lower bounds to the Gromov-Hausdorff distance as shown in \cite{Mem12}. 
We refer to the latter paper for a wide range of stable invariants. Additional examples are 
hierarchical clustering (\cite{CarMem1,CarMem2}) and persistence diagrams induced by the Vietoris-Rips, the Dowker, and the \v{C}ech filtrations (\cite{Cha_al,Cha_al2}, see also \cite{EdeHar} for details and applications of persistent homology).

In contrast to false negatives, even though still undesirable, it is often acceptable when an invariant produces false positives, where two dissimilar spaces are mapped to close values
. In this paper, we study when stable invariants are actually bound to lose information because of the unavoidable creation of false positives. We focus our study on those invariants taking values in a Hilbert space. Those are particularly relevant for the applications in machine learning pipelines since many algorithms expect either data in the form of Euclidean vectors or at least access to a so-called feature map into a Hilbert space. 
As formally stated in the sequel, we prove that the existence of a stable invariant avoiding false positives strongly depends on a bound on the cardinality of the metric spaces.

Our approach to the problem requires notions and techniques from coarse geometry. Intuitively, this field, also known as large-scale geometry, focuses on large-scale, global properties of spaces ignoring local features. We refer to \cite{Roe,NowYu} for a wide introduction. 
A map $\psi\colon (X,d_X)\to (Y,d_Y)$ between two metric spaces is said to be 
a {\em coarse embedding} if there exist two maps $\rho_-,\rho_+\colon\R_{\geq 0}\to\R_{\geq 0}$ such that $\rho_-\to\infty$ and, for every $x,y\in X$,
	\begin{equation}\label{eq:rho-_rho+}\rho_-(d_X(x,y))\leq d_Y(\psi(x),\psi(y))\leq\rho_+(d_X(x,y)).\end{equation} 
In the case of stable invariants, i.e., satisfying \eqref{eq:borno}, a lower bound as in \eqref{eq:rho-_rho+} prevents false positives since the larger the Gromov-Hausdorff distance, the larger the distance between the two associated invariants.

Coarse embeddings have been introduced by Gromov and extensively studied in coarse geometry. 
A crucial application of this theory is due to Yu, who proved in \cite{Yu_CE} that those metric spaces that can be coarsely embedded into a Hilbert space satisfy the Novikov and the coarse Baum-Connes conjectures generalising results contained in \cite{Yu_asdim}. This result motivated two research directions. On one hand, since an explicit coarse embedding can be hard to construct, a plethora of conditions ensuring its existence have been defined and investigated. We refer the interested reader to \cite{NowYu} for a discussion on the topic and to \cite{WeiYamZav} for more examples. Among these properties, if a space has finite asymptotic dimension, then it can be coarsely embedded into a Hilbert space (\cite{Yu_CE,HigRoe}). Asymptotic dimension is a large-scale counterpart of Lebesgue's covering dimension introduced in \cite{Gro_asdim} (see also \cite{BelDra}). On the other hand, examples of metric spaces that cannot be coarsely embedded were constructed 
for example in \cite{DraGonLafYu,Laf}
. 
Showing that one of those pathological examples can be coarsely embedded into a metric space $X$ is a technique to prove that $X$ itself cannot be coarsely embedded into any Hilbert space.

Those two strategies have been adopted to prove if metric spaces emerging in different fields can be coarsely embedded into Hilbert spaces. In topological data analysis, collections of persistence diagrams endowed with various metrics represent a prominent example. 
It was proved in \cite{MitVir} that the space of persistence diagrams of at most $n$ points endowed with the Hausdorff distance has finite asymptotic dimension, and so it can be coarsely embedded into a Hilbert space. This result, despite being non-constructive, motivated further research in that direction that eventually led to explicit bi-Lipschitz and coarse embeddings in \cite{BatGar} and \cite{MitVir_ce}, respectively. In the opposite direction, it was proved in \cite{BubWag,Wag,MitVir} that spaces of persistence diagrams with various metrics cannot be coarsely embedded into any Hilbert space. We also refer to \cite{PriWei}, where the authors showed the equivalence of this problem with the embeddability of Wasserstein spaces. 

Another example can be found in \cite{GarVirZav}. Motivated by the interest in invariants in crystallography and pharmaceutics (see \cite{EdeHeiKurSmiWin}), 
the authors used the previously described strategy to prove that 
spaces of periodic point sets equipped with the Euclidean bottleneck distance cannot be coarsely embedded into any uniformly convex Banach space. The constructions used in that paper are adapted from those developed in \cite{WeiYamZav} to prove the analogous non-coarse embeddability results for families of finite subsets of metric spaces endowed with the Hausdorff distance.

In this paper, we prove the following results.
\begin{theoremintro}\label{theo:theoA}
The space $\mathcal{GH}^{\leq n}$ of isometry classes of metric spaces with at most $n$ points endowed with the Gromov-Hausdorff distance has asymptotic dimension $n(n-1)/2$, and so it can be coarsely embedded into a Hilbert space.
\end{theoremintro}
\begin{theoremintro}\label{theo:theoB}
The space $\mathcal{GH}^{<\omega}$ of isometry classes of finite metric spaces endowed with the Gromov-Hausdorff distance cannot be coarsely embedded into any uniformly convex Banach space, and so, in particular, into any Hilbert space.
\end{theoremintro}
As an immediate consequence of Theorem \ref{theo:theoB}, the same result holds for the {\em Gromov-Hausdorff space} $\mathcal{GH}$, which is the metric space of isometry classes of compact metric spaces equipped with the Gromov-Hausdorff distance. In the paper, we prove a stronger version of Theorem \ref{theo:theoB} stating that already the much smaller subspace of $\mathcal{GH}^{<\omega}$ consisting of all isometry classes of finite subsets of the real line cannot be coarsely embedded into any uniformly convex Banach space (Theorem \ref{thm:EH_not_emb}). Thanks to a recent result due to Majhi, Vitter and Wenk (\cite{MajVitWen}), Theorem \ref{thm:EH_not_emb}, and therefore also Theorem \ref{theo:theoB}, will follow from the fact that an obstruction to coarse embeddability is found in 
the space of isometry classes of finite subsets of the real line endowed with the Euclidean-Hausdorff distance (a modification of the Gromov-Hausdorff distance for subsets of $\R^d$).

We conclude the paper focusing on the subspace $\mathcal{GH}_{\leq R}^{<\omega}$ of $\mathcal{GH}^{<\omega}$ whose elements have diameter bounded by a constant $R>0$. Since the diameter of this subspace is bounded, the map collapsing the space into a point is trivially a coarse embedding. An immediate follow-up question is whether it can be bi-Lipschitz embedded, as in the case of persistence diagrams with at most $n$ points. We provide a partial negative answer.

\begin{theoremintro}\label{theo:theoC}
$\mathcal{GH}_{\leq R}^{<\omega}$ cannot be bi-Lipschitz embedded into any finite-dimensional Hilbert space.
\end{theoremintro}
Inspired by the paper \cite{CarBau}, where the authors proved that certain spaces of persistence diagrams cannot be bi-Lipschitz embedded into any finite-dimensional Hilbert space, we compute the Assouad dimension (\cite{Ass}, see also \cite{Bou} for an earlier definition) of $\GHs_{\leq R}^{<\omega}$ and show that it is infinite. This dimension notion was in fact introduced to provide such embeddability obstructions. More precisely, we show that already the subset consisting of all isometry classes of finite subsets of an interval has infinite Assouad dimension and cannot be bi-Lipschitz embedded into any finite-dimensional Hilbert space. Again, using the aforementioned Majhi, Vitter and Wenk's theorem, we deduce our claims from the analogous results for the space of isometry classes of finite subsets of an interval endowed with the Euclidean-Hausdorff distance (Theorem \ref{thm:EH_leq_R_non_embed} and Proposition \ref{prop:dimA}).

\smallskip

The paper is organised as follows. In Section \ref{sec:GH_EH} we provide the needed background regarding the Gromov- and the Euclidean-Hausdorff distances. In Section \ref{sec:asdim}, the asymptotic dimension is introduced and Theorem \ref{theo:theoA} is proved. Theorem \ref{theo:theoB} is shown in Section \ref{sec:non_coarse_emb}, and, finally, we define the Assouad dimension and provide Theorem \ref{theo:theoC} in Section \ref{sec:non_bi_lip_emb}. We conclude the paper discussing a list of questions in \S\ref{sub:q}.


\smallskip

{\bf Acknowledgements.} The author was supported by the FWF Grant, Project number I4245-N35. 
The author would like to thank Thomas Weighill for the helpful discussions around Theorem \ref{theo:coarsely_k_to_1_appl}, and Takamitsu Yamauchi for bringing to my attention the fundamental reference \cite{IliIvaTuz}. Furthermore, the author is thankful for the detailed and helpful comments of the reviewer of this manuscript. 

\smallskip

{\bf Notation.} 
We denote by $\N$, $\Q$ and $\R$ the set of natural numbers including $0$, the set of rational numbers, and the set of real numbers, respectively. For $c\in\R$, we also write 
$$\R_{\geq c}=\{x\in\R\mid x\geq c\},\quad\text{and}\quad\R_{>c}=\{x\in\R\mid x>c\}.$$

For a set $X$, we denote by $\lvert X\rvert$ its cardinality. Moreover, for $n\in\N$, we define the following subsets of the power set of $X$: 
$$[X]^{=n}=\{A\subseteq X\mid\lvert A\rvert= n\},\quad[X]^{\leq n}=\bigcup_{k\leq n}[X]^{=n},\quad\text{and}\quad[X]^{<\omega}=\bigcup_{k\in\N}[X]^{=n}.$$


\section{The Gromov-Hausdorff distance and the Euclidean-Hausdorff distance}\label{sec:GH_EH}

We recall some basic notions, the definitions of the Gromov-Hausdorff and the Euclidean-Hausdorff distances and their relationships. We refer to \cite{BurBurIva,Pet,Tuz,Mem,Mem12} for comprehensive discussions on the Gromov-Hausdorff distance.

\begin{definition}\label{def:distance}
A pair $(X,d)$ consisting of a set $X$ and a map $d\colon X\times X\to\R$ is called a {\em network} (\cite{ChoMem3}). A network $(X,d)$ is a {\em metric space} (and $d$ is a {\em metric}) if it satisfies the following properties:
\begin{compactenum}
\item[(M1)] for every $x,y\in X$, $d(x,y)\geq 0$ and $d(x,x)=0$;
\item[(M2)] for every $x,y\in X$, $d(x,y)=0$ if and only if $x=y$;
\item[(M3)] for every $x,y\in X$, $d(x,y)=d(y,x)$;
\item[(M4)] for every $x,y,z\in X$, $d(x,y)\leq d(x,z)+d(z,y)$.
\end{compactenum}
\end{definition}
Let us recall that an {\em isometry} between two networks $(X,d_X)$ and $(Y,d_Y)$ is a bijective map $\psi\colon X\to Y$ such that, for every $x,x^\prime\in X$, $d_Y(\psi(x),\psi(x^\prime))=d_X(x,x^\prime)$. In that case, $X$ and $Y$ are said to be {\em isometric}.

For a subset $A$ of a metric space $(X,d)$, its {\em diameter} is the value
$$\diam A=\sup_{x,y\in A}d(x,y).$$

A {\em correspondence $\mathcal R$} between two sets $X$ and $Y$ is a relation $\mathcal R\subseteq X\times Y$ such that every $x\in X$ is in relation with at least one element $y\in Y$ and vice versa. Then, for every metric space $(Z,d)$, the {\em Hausdorff distance} is defined as follows: for every $X,Y\subseteq Z$,
$$d_H(X,Y)=\inf_{\mathcal R\subseteq X\times Y\text{ correspondence}}\sup_{(x,y)\in\mathcal R}d(x,y).$$

\begin{definition}\label{def:GH}
Given two metric spaces $X$ and $Y$, their {\em Gromov-Hausdorff distance} $\GH$ is the value
$$\GH(X,Y)=\inf_{\text{$Z$ metric space}}\inf\{d_H(i_X(X),i_Y(Y))\mid i_X\colon X\to Z\text{ and }i_Y\colon Y\to Z\text{ isometric embeddings}\}.$$
\end{definition}
The reader may notice an abuse of notation in the previous definition since all possible metric spaces form a proper class. However, the infimum value can be achieved by investigating just a set of spaces. Indeed, it is enough to consider the disjoint union $X\sqcup Y$ endowed with pseudo-metrics (where the distance between distinct points may be zero) whose restrictions to the subsets $X$ and $Y$ coincide with the original metrics. We refer to \cite{BurBurIva} for the details. 

If two metric spaces are isometric, their Gromov-Hausdorff distance is $0$. The converse implication does not hold in general. However, if $X$ and $Y$ are compact and $\GH(X,Y)=0$, then $X$ and $Y$ are isometric.


Denote by $\GHs$ the set of all isometry classes of compact metric spaces endowed with $\GH$, where the Gromov-Hausdorff distance between two isometry classes is the Gromov-Hausdorff distance between any pair of representatives. Since two compact metric spaces are isometric if and only if their Gromov-Hausdorff distance is $0$ (see, for example, \cite{BurBurIva}), $\GHs$ is a metric space, also called the {\em Gromov-Hausdorff space}. Furthermore, we consider the subspace $\GHs^{<\omega}$ of $\GHs$ consisting of isometry classes of finite metric space, which is dense in $\GHs$. Actually, the subspace $\GHs^{<\omega}_\Q$ of isometry classes of finite spaces endowed with metrics taking values in $\Q$ is dense in $\GHs$ (\cite{Pet}). Therefore, $\GHs$ is separable.

The Gromov-Hausdorff distance can be alternatively characterised using correspondences. 
If $(X,d_X)$ and $(Y,d_Y)$ are two networks, 
and $\mathcal R\subseteq X\times Y$ is a correspondence between them, the {\em distortion of $\mathcal R$} is the value
$$\dis\mathcal R=\sup_{(x_1,y_1),(x_2,y_2)\in\mathcal R}\lvert d_X(x_1,x_2)-d_Y(y_1,y_2)\rvert.$$
\begin{definition}[\cite{ChoMem3}]
Let $X$ and $Y$ be two networks, then their {\em network distance} is the value
$$d_{\mathcal N}(X,Y)=\frac{1}{2}\inf_{\text{$\mathcal R\subseteq X\times Y$ correspondence}}\dis\mathcal R.$$
\end{definition}
It is known that, if $X$ and $Y$ are metric spaces, then $\GH(X,Y)=d_{\mathcal N}(X,Y)$ (see, for example, \cite{BurBurIva}). 
A further characterisation of the Gromov-Hausdorff distance can be found in \cite{KalOst}.


%
In \cite{Zav}, a characterisation of the network distance for {\em quasi-metric spaces} (i.e., networks satisfying (M1), (M2) and (M4)) in the spirit of Definition \ref{def:GH} is provided.

The Gromov-Hausdorff distance is difficult to compute even in simple cases. For example, the distance between spheres of different dimensions endowed with their geodesic distance is not known in general (\cite{LimMemSmi,AdaBusCla+}). To approximate it, it is convenient to consider another related distance.
\begin{definition}
Let $X$ and $Y$ be two subsets of $\R^d$. Consider them as metric spaces. Then, their {\em Euclidean-Hausdorff distance} $\EH$ is defined as follows:
$$\EH(X,Y)=\inf\{d_H(i_X(X),i_Y(Y))\mid i_X\colon X\to\R^d\text{ and }i_Y\colon Y\to\R^d\text{ isometric embeddings}\}.$$
\end{definition}
Using the following folklore result (see, for example, \cite[Ch. IV, \S38]{Blu}), $\EH$ can be conveniently characterised. 
\begin{theorem}
If $f\colon X\to Y$ is an isometry between two subsets of $\R^d$, then there exists an isometry $\widetilde f\colon\R^d\to\R^d$ such that  $\widetilde f|_X=f$.  
\end{theorem}

\begin{corollary}{\rm (see, for example, \cite[Corollary  4.3]{Ant1})} 
If $X$ and $Y$ are two subsets of $\R^d$, then
$$\EH(X,Y)=\inf_{f\in\Iso(\R^d)}d_H(X,f(Y)),$$
where $\Iso(\R^d)$ denotes the group of isometries of $\R^d$.
\end{corollary}


Clearly, $\EH(X,Y)\geq\GH(X,Y)$ for every pair $X$ and $Y$ of subsets of an euclidean space $\R^d$. Moreover, the inequality can be strict (for example, see \cite{Mem}). A lower bound on the Gromov-Hausdorff distance depending on the Euclidean-Hausdorff distance was provided in \cite{Mem}.
\begin{theorem}
For every pair of compact subsets $X$ and $Y$ of $\R^d$,
$$\GH(X,Y)\leq\EH(X,Y)\leq c_d\sqrt{M\cdot\GH(X,Y)},$$
where $M=\max\{\diam X,\diam Y\}$ and $c_d$ is a constant depending only on the dimension $d$.
\end{theorem}
However, if $X$ and $Y$ are two finite subsets of $\R$, a linear lower bound to $\GH$ depending on $\EH$ can be proved.	
\begin{theorem}{\rm (\cite[Theorem 3.2]{MajVitWen})}\label{thm:biLipschitz}
 For every pair $X$ and $Y$ of compact subsets of $\R$,
	$$\frac{4}{5}\EH(X,Y)\leq\GH(X,Y)\leq\EH(X,Y).$$
\end{theorem}
If we denote  by $\mathcal{EH}_1$ ($\EHs_1^{<\omega}$) the space of isometry classes of compact (finite, respectively) subsets of the real line endowed with the Euclidean-Hausdorff distance, the canonical inclusion of $\mathcal{EH}_1$ into $\GHs$ and that of $\EHs_1^{<\omega}$ into $\GHs^{<\omega}$ 
are bi-Lipschitz 
according to Theorem \ref{thm:biLipschitz}. Let us recall that a map $\psi\colon X\to Y$ between metric spaces is a {\em bi-Lipschitz embedding} if there are two linear maps $\rho_-\colon x\mapsto a\cdot x$ and $\rho_+\colon x\mapsto b\cdot x$, where $a,b>0$, satisfying the condition \eqref{eq:rho-_rho+}.

\section{The space of metric spaces of at most $n$ points is coarsely embeddable into a Hilbert space}\label{sec:asdim}

Given two subsets $Y,Z$ of a metric space $(X,d)$ and a radius $r\geq 0$ we write
$$\dist_d(Y,Z)=\dist(Y,Z)=\inf\{d(y,z)\mid y\in Y, z\in Z\},\quad\text{and}\quad B_d(Y,r)=\bigcup_{y\in Y}B_d(y,r),$$
where $B_d(y,r)$ denotes the closed ball centred in $y$ with radius $r$.

A family of subsets $\mathcal U$ of a metric space $X$ is said to be 
\begin{compactenum}[$\bullet$]
\item {\em uniformly bounded} if there exists $R\geq 0$ such that $\diam U\leq R$ for every $U\in\mathcal U$ (if we need to specify the value $R$, we say it is {\em $R$-bounded});
\item {\em $r$-disjoint} for some $r>0$ if $\dist(U,V)>r$ for every $U,V\in\mathcal U$ with $U\neq V$.
\end{compactenum}

\begin{definition}[\cite{Gro_asdim}]
Let $X$ be a metric space. The {\em asymptotic dimension} of $X$ is at most $n\in\N$ (and we write $\asdim X\leq n$) if for every $r>0$ there exists a uniformly bounded cover $\mathcal U=\mathcal U_0\cup\cdots\cup\mathcal U_n$ of $X$ such that $\mathcal U_i$ is $r$-disjoint for every $i=0,\dots,n$. Furthermore, $\asdim X=n$ if $\asdim X\leq n$ and $\asdim X\not\leq n$, and $\asdim X=\infty$ if 
$\asdim X\not\leq m$ for every $m\in\N$.
\end{definition}

\begin{example}[see \cite{NowYu}]
For every $n\in\N$, $\asdim\R^n=\asdim(\R_{\geq 0})^n=n$ where $\R^n$ is equipped with any $p$-norm, $p\in[1,\infty]$, and $(\R_{\geq 0})^n$ with any of the inherited metrics.
\end{example}
Let us recall two basic properties of the asymptotic dimension that we are going to use later in this section.
\begin{proposition}[see \cite{BelDra}]\label{prop:asdim_monotone}
Let $(X,d)$ be a metric space and $Y\subseteq X$ be a subspace. Then
\begin{compactenum}[(a)]
\item $\asdim Y\leq\asdim X$, and
\item $\asdim Y=\asdim X$ provided that $Y$ is {\em large} in $X$, i.e., there exists $r\geq 0$ such that $B_d(Y,r)=X$.
\end{compactenum}
\end{proposition}

The goal of this section is to prove Theorem \ref{theo:theoA}, which we obtain as a particular case of the more general Corollary \ref{coro:asdimGHn}.

\begin{lemma}\label{lemma:asdim_geq}
For every $n\in\N$, $\asdim\GHs^{\leq n}\geq n(n-1)/2.$    
\end{lemma}
\begin{proof}
According to \cite[Theorem 4.1]{IliIvaTuz}, $\GHs^{\leq n}$ contains isometric copies of arbitrarily large balls of $\R^{n(n-1)/2}$ endowed with the supremum metric. Then, \cite[Lemma 2.10]{MitVir} implies that $\asdim\GHs^{\leq n}\geq n(n-1)/2$.	
\end{proof}
In order to prove the opposite inequality, we need to show different steps. Let us start with recalling a known result.

\begin{theorem}\label{lemma:Rad_Shu}{\rm (\cite{RadShu}, see also \cite[Theorem 4.6]{KucZar})}
	For every $n\in\N$, $\asdim([X]^{\leq n},d_H)\leq n\asdim X$. In particular, $\asdim[\R]^{\leq n}\leq n$ and $\asdim[\R_{\geq 0}]^{\leq n}\leq n$.
\end{theorem}
	
For every positive integer $n\in\N$, let us define $\mathcal X_n$ as the metric subspace of $([\R_{\geq 0}]^{\leq n+1},d_H)$ whose elements contain the point $0$. According to Theorem \ref{lemma:Rad_Shu}, $\asdim\mathcal X_n\leq n+1$ since the asymptotic dimension is monotone (Proposition \ref{prop:asdim_monotone}(a)). In the sequel, we improve that bound showing that $\asdim\mathcal X_n\leq n$. The proof outline is similar to and inspired by that of \cite[Theorem 3.2]{MitVir}. 
First, we need a classical preliminary result.

Given two families $\mathcal U$ and $\mathcal V$ of subsets of a metric space $X$ and $r>0$, we define a new family of subsets as follows:
\begin{gather*}\mathcal U\cup_r\mathcal V=\{N_r(U,V)\mid U\in\mathcal U\}\cup\{V\in\mathcal V\mid\forall U\in\mathcal U,\,\dist(V,U)>r\}\\
\text{where}\quad N_r(U,\mathcal V)=U\cup\bigcup\{V\in\mathcal V\mid\dist(U,V)\leq r\}.\end{gather*}
Let us immediately note that $\bigcup(\mathcal V\cup_r\mathcal U)\supseteq\bigcup\mathcal V\cup\bigcup\mathcal U$.

\begin{lemma}\label{lemma:BelDra}{\rm (\cite[Proposition 24]{BelDra})}
Let $\mathcal U$ be an $r$-disjoint, $R$-bounded family of subsets of $X$ with $R\geq r$. Let $\mathcal V$ be a $5R$-disjoint, uniformly bounded family of subsets of $X$. Then $\mathcal V\cup_r\mathcal U$ is $r$-disjoint and uniformly bounded.
\end{lemma}

\begin{lemma}\label{lemma:asdimX_n}
	For every $n\in\N$, $\asdim\mathcal X_n\leq n$.
\end{lemma}
\begin{proof}
	Let us prove the result by induction.
	
	If $n=1$, then, for every $r>0$, the usual uniformly bounded cover used to prove that the real line has asymptotic dimension $1$ (see \cite{BelDra,NowYu}) can be adapted. More precisely, for every $r>0$, define, for $i=0,1$,
	$$\mathcal U_i=\{\{\{0,x\}\mid x\in V_k^i\}\mid k\in\N\},\quad\text{where}\quad V_k^i=[(4k+2i)r,(4k+2i+2)r]\subseteq\R_{\geq 0}$$
(see Figure \ref{fig:asdim_R_leq_1}). Then, each $\mathcal U_i$ is $r$-disjoint and $\mathcal U_0\cup\mathcal U_1$ forms a uniformly bounded cover of $\mathcal X_1$. Hence, $\asdim\mathcal X_1\leq 1$.

 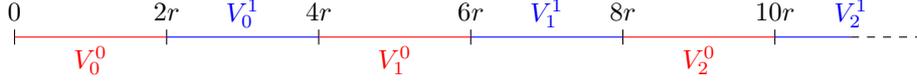
\begin{figure}
     \centering
     \begin{tikzpicture}
     \draw (0,-0.1)--(0,0.1)node[above]{$0$};
    \draw (2,-0.1)--(2,0.1)node[above]{$2r$};
     \draw (4,-0.1)--(4,0.1)node[above]{$4r$};
     \draw (6,-0.1)--(6,0.1)node[above]{$6r$};
     \draw (8,-0.1)--(8,0.1)node[above]{$8r$};
     \draw (10,-0.1)--(10,0.1)node[above]{$10r$};
     \draw[red] (1,0)node[below]{$V_0^0$};
     \draw[blue] (3,0)node[above]{$V_0^1$};
    \draw[red] (5,0)node[below]{$V_1^0$};
     \draw[blue] (7,0)node[above]{$V_1^1$};
    \draw[red] (9,0)node[below]{$V_2^0$};
     \draw[blue] (11,0)node[above]{$V_2^1$};
     \draw[dashed] (11,0)--(12,0);
     \draw[red](0,0)--(2,0) (4,0)--(6,0) (8,0)--(10,0);
     \draw[blue] (2,0)--(4,0) (6,0)--(8,0) (10,0)--(11,0);
     \end{tikzpicture}
     \caption{
     A representation of the uniformly bounded cover $\mathcal V=\mathcal V_0\cup\mathcal V_1$ showing that $\asdim\R_{\geq 0}\leq 1$. In red, the elements of $\mathcal V_0=\{V^0_k\mid k\in\N\}$ and the subsets contained in the family $\mathcal V_1=\{V_k^1\mid k\in\N\}$ are in blue.}
     \label{fig:asdim_R_leq_1}
 \end{figure}
	
	Suppose now that the assertion is true for some $n\in\N$. We want to show that $\asdim\mathcal X_{n+1}\leq n+1$. Fix a positive $r>0$. Then we can decompose
	$$\mathcal X_{n+1}=\widetilde{\mathcal X_{n+1}}\cup B_{d_H}(\mathcal X_n,r),$$
	where $\widetilde{\mathcal X_{n+1}}$ consists of subsets $A$ of $\R_{\geq 0}$ such that, if $x\in A$ has $\lvert x\rvert\leq r$, then $x=0$. It indeed represents a partition. In fact, every element $C\in\mathcal X_{n+1}$ can be decomposed $C=C_0\cup C_1$ where $\max_{x\in C_0}\lvert x\rvert\leq r$ and $\min_{x\in C_1}\lvert x\rvert>r$. Then, either $C_0=\{0\}$ and so $C\in\widetilde{\mathcal X_{n+1}}$, or $d_H(C,C_1\cup\{0\})\leq r$ and $C_1\cup\{0\}\in\mathcal X_n$.
	
	Let us now consider $\widetilde{\mathcal X_{n+1}}$, and define the subspace $\mathcal Y_n=\{X\setminus\{0\}\mid X\in\widetilde{\mathcal X_{n+1}}\}$ of $[\R_{\geq 0}]^{\leq n+1}$. Monotonicity of the asymptotic dimension (Proposition \ref{prop:asdim_monotone}(a)) implies that $\asdim\mathcal Y_n\leq n+1$. Hence, there exists a uniformly bounded cover $\mathcal W=\mathcal W_0\cup\cdots\cup\mathcal W_{n+1}$ of $\mathcal Y_n$ such that $\mathcal W_i$ is $r$-disjoint for every $i=0,\dots,n+1$. Let $R\geq r$ be an upper bound to the diameter of the elements in $\mathcal W$. For every $i=0,\dots,n+1$, construct the family of subsets $\widetilde{\mathcal W_i}=\{W\cup\{0\}\mid W\in\mathcal W_i\}$ of $\widetilde{\mathcal X_{n+1}}$. Note that $\widetilde{\mathcal W}=\widetilde{\mathcal W_0}\cup\cdots\cup \widetilde{\mathcal W_{n+1}}$ is a cover of $\widetilde{\mathcal X_{n+1}}$. Furthermore, each element of $\widetilde{\mathcal W}$ has diameter bounded by $R$ and it is easy to see that each of the families $\widetilde{\mathcal W_i}$, $i=0,\dots,n+1$, is $r$-disjoint.
	
	Since $\mathcal X_n$ is large in $B_{d_H}(\mathcal X_n,r)$, $\asdim B_{d_H}(\mathcal X_n,r)\leq n\leq n+1$ by the inductive hypothesis and Proposition \ref{prop:asdim_monotone}(b). Hence, there exists a uniformly bounded cover $\mathcal V=\mathcal V_0\cup\cdots\cup\mathcal V_{n+1}$ of $B_{d_H}(\mathcal X_n,r)$ such that $\mathcal V_i$ is $5R$-disjoint for every $i=0,\dots,n+1$. 
	
	Define, for every $i=0,\dots,n+1$,
	$$\mathcal U_i=\widetilde{\mathcal W_i}\cup_r\mathcal V_i.$$
	Hence, $\mathcal U=\mathcal U_0\cup\cdots\cup\mathcal U_{n+1}$ is a uniformly bounded cover of $\mathcal X_{n+1}=\bigcup\widetilde{\mathcal W}\cup\bigcup\mathcal V$ and $\mathcal U_i$ is $r$-disjoint for every $i=0,\dots,n+1$ according to Lemma \ref{lemma:BelDra}. Thus, $\asdim\mathcal X_{n+1}\leq n+1$.
\end{proof}

On the family of all isometry classes of finite networks, the network distance $d_{\mathcal N}$ is a {\em pseudo-metric}, i.e., it satisfies the properties (M1), (M3) and (M4) stated in Definition \ref{def:distance} (\cite{ChoMem3}). Consider the equivalence relation between finite networks given by $X\sim Y$ if $d_{\mathcal N}(X,Y)=0$. We refer to \cite{ChoMem23} where this equivalence relation is completely characterised. The space obtained by quotienting the family of all isometry classes of finite networks under this equivalence relation is a metric space (it is a standard way to obtain a metric space out of a pseudo-metric space, see for example \cite[Proposition 1.1.5]{BurBurIva}). We denote it by $\mathcal N^{<\omega}$. Let us recall that the distance between two elements in $\mathcal N^{<\omega}$ is the network distance between any two representatives. For the sake of simplicity, we will be directly working with representatives of objects in $\mathcal N^{<\omega}$ instead of their equivalence classes without explicit mention. 
Let us also consider the following metric subspaces of $\mathcal N^{<\omega}$:
\begin{compactenum}[$\bullet$]
\item $\mathcal S^{<\omega}$ -- the family of equivalence 
classes of finite {\em pseudo-semi-metric spaces}, i.e., networks satisfying (M1) and (M3);
\item $\mathcal N^{\leq n}$ -- the family of equivalence classes of 
networks whose cardinality is at most $n$;
\item $\mathcal S^{\leq n}=\mathcal S^{<\omega}\cap\mathcal N^{\leq n}$.
\end{compactenum}

Define the map $\mathcal D\colon\mathcal N^{<\omega}\to[\R]^{<\omega}$ as follows: for every $(X,d)\in\mathcal N^{<\omega}$,
$$\mathcal D(X)=d(X\times X)\subseteq\R.$$
The subset $\mathcal D(X)$ is called the {\em distance set} of $X$. Note that $\mathcal D|_{\mathcal S^{\leq n}}\colon \mathcal S^{\leq n}\to\mathcal X_{n(n-1)/2}$. For the sake of simplicity, we denote by $\mathcal D$ also the restriction.

\begin{proposition}{\rm(\cite[Proposition 4.3.4]{ChoMem23})}\label{prop:distance_set_lipschitz}
The map $\mathcal D\colon\mathcal N^{<\omega}\to([\R]^{<\omega},d_H)$ is well-defined and {\em $2$-Lipschitz} (i.e., $d_H(\mathcal D(X),\mathcal D(Y))\leq 2d_{\mathcal N}(X,Y)$ for every $X,Y\in\mathcal N^{<\omega}$).    
\end{proposition}
\begin{proof}
Let $(X,d_X),(Y,d_Y)\in\mathcal N^{<\omega}$ and $\mathcal R$ be a correspondence such that $R=\dis\mathcal R=2d_{\mathcal N}(X,Y)$. For every $x_1,x_2\in X$, pick $y_1,y_2\in Y$ satisfying $(x_i,y_i)\in\mathcal R$, for $i=1,2$. Then, $\lvert d_X(x_1,x_2)-d_Y(y_1,y_2)\rvert\leq R$. Since the same argument can be carried out also for every pair of points $y_1,y_2\in Y$, 
$d_H(\mathcal D(X),\mathcal D(Y))\leq R$. This inequality also implies that the map is well-defined.
\end{proof}
Even if we restrict ourselves to consider only metric spaces that are finite subsets of the real line, the map $\mathcal D$ is not injective (\cite{Blo}). However, it has a further property that allows us to deduce an upper bound to the asymptotic dimension of $\GHs^{\leq n}$.

Let us recall that, for a positive integer $k\in\N$, a map $\varphi\colon(X,d_X)\to(Y,d_Y)$ between metric space is {\em coarsely $k$-to-$1$} (\cite{MiyVir}) if the following two properties hold:
\begin{compactenum}[$\bullet$]
\item there exists a map $\rho_+\colon\R_{\geq 0}\to\R_{\geq 0}$ such that, for every $x,y\in X$,
$$d_Y(\varphi(x),\varphi(y))\leq \rho_+(d_X(x,y))$$
(i.e., $\varphi$ is {\em bornologous});
\item for every $R\geq 0$, there exists $S\geq 0$ such that for every $y\in Y$ there are $x_1,\dots,x_k\in X$ satisfying
$$\varphi^{-1}(B_{d_Y}(y,R))\subseteq \bigcup_{i=1}^kB_{d_X}(x_i,S).$$
\end{compactenum}
Coarsely $k$-to-$1$ maps play a very important role in providing bounds for the asymptotic dimension.
\begin{theorem}[\cite{MiyVir,DydVir}, see also \cite{DydWei}]\label{theo:coarsely_k_to_1}
Let $\varphi\colon X\to Y$ be a coarsely $k$-to-$1$ surjective map between metric spaces. Then,
$$\asdim X\leq\asdim Y\leq k(\asdim X+1)-1.$$
\end{theorem}

\begin{theorem}\label{theo:coarsely_k_to_1_appl}
The map $\mathcal D\colon \mathcal S^{\leq n}\to\mathcal D(\mathcal S^{\leq n})=\mathcal X_{n(n-1)/2}
$ is coarsely $k$-to-$1$ for some suitable $k$.
\end{theorem}
\begin{proof}
Proposition \ref{prop:distance_set_lipschitz} implies that the map is bornologous.

Let us fix $R\in\R_{\geq 0}$ and $D\in\mathcal D(\mathcal S^{\leq n})=\mathcal X_{n(n-1)/2}$. In particular, $0\in D$. Suppose that $(Y,d_Y)\in\mathcal S^{\leq n}$ satisfies $d_H(D,\mathcal D(Y))\leq R$. Then, there exists a function $f\colon\mathcal D(Y)\to D$ such that $\lvert f(a)-a\rvert\leq R$ for every $a\in\mathcal D(Y)$. Furthermore, without loss of generality, we can require that $f(0)=0$ (note that $0\in\mathcal D(Y)$). 

Define a new object $X_Y\in\mathcal S^{\leq n}$ as follows: $X_Y=(Y,f\circ d_Y)$. The following properties hold:
\begin{compactenum}[(a)]
\item $X_Y\in\mathcal S^{\leq n}$;
\item $\mathcal D(X_Y)\subseteq D$;
\item $d_{\mathcal N}(Y,X_Y)\leq\frac{1}{2}\dis id\leq R/2$.
\end{compactenum}

Items (a) and (b) imply that
$$\lvert\{X_Y\mid Y\in\mathcal S^{\leq n}:d_{H}(\mathcal D(Y),D)\leq R\}\rvert\leq\lvert D\rvert^{\frac{n(n-1)}{2}}\leq \bigg(\frac{n(n-1)}{2}+1\bigg)^{\frac{n(n-1)}{2}}=:k.$$
Hence, $\mathcal D$ is a coarsely $k$-to-$1$ map.
\end{proof}

\begin{corollary}\label{coro:asdimGHn}
For every $n\in\N$, $\asdim\GHs^{\leq n}=\asdim\mathcal S^{\leq n}=n(n-1)/2$.
\end{corollary}
\begin{proof}
The claim is implied by the following chain of inequalities:
$$\frac{n(n-1)}{2}\leq\asdim\GHs^{\leq n}\leq\asdim\mathcal S^{\leq n}\leq\asdim\mathcal D(\mathcal S^{\leq n})=\asdim\mathcal X_{\frac{n(n-1)}{2}}\leq\frac{n(n-1)}{2},$$
where the first inequality is stated in Lemma \ref{lemma:asdim_geq}, the second one follows from Proposition \ref{prop:asdim_monotone}(a), the third one from Theorem \ref{theo:coarsely_k_to_1}, and the last from Lemma \ref{lemma:asdimX_n}.
\end{proof}

Thus, Theorem \ref{theo:theoA} follows since having finite asymptotic dimension implies the existence of a coarse embedding into a Hilbert space (\cite{Yu_CE,HigRoe}). 


\begin{remark}
Similarly, we can show that $\mathcal D\colon\mathcal N^{\leq n}\to[\R]^{n^2}$ is 
coarsely $(n^2)^{n^2}$-to-$1$ surjective map and so $\asdim\mathcal N^{\leq n}\leq\asdim[\R]^{\leq n^2}\leq n^2$ by Theorem \ref{theo:coarsely_k_to_1} and Lemma \ref{lemma:Rad_Shu}. Hence, $n(n-1)/2\leq\asdim\mathcal N^{\leq n}\leq n^2$. 
In particular, $\mathcal N^{\leq n}$ can be coarsely embedded into a Hilbert space.
\end{remark}    


\section{Coarse non-embeddability into Hilbert spaces}\label{sec:non_coarse_emb}

Let us now recall some definitions and results coming from coarse geometry as they will be the crucial stepping stones to prove our main results.

Suppose that the map $\psi\colon(X,d_X)\to(Y,d_Y)$ between metric spaces is a coarse embedding. If $\rho_-\colon\R_{\geq 0}\to\R_{\geq 0}$ and $\rho_+\colon\R_{\geq 0}\to\R_{\geq 0}$ are two maps such that $\rho_-\to\infty$ and \eqref{eq:rho-_rho+} is fulfilled, then we call $\rho_-$ and $\rho_+$ {\em control functions}, and we say that $\psi$ is a {\em $(\rho_-,\rho_+)$-coarse embedding}. 

Let $\{X_k\}_{k\in\N}$ be a sequence of metric spaces and $X$ be another metric space. We say that a family of maps $\{i_k\colon X_k\to X\}_{k\in\N}$ is a {\em coarse embedding} if there exist two maps $\rho_-,\rho_+\colon\R_{\geq 0}\to\R_{\geq 0}$ such that $i_k$ is a $(\rho_-,\rho_+)$-coarse embedding for every $k\in\N$.
We say that $X$ {\em contains a coarse disjoint union of $\{X_k\}_{k\in\N}$} if there exists a coarse embedding $\{i_k\colon X_k\to X\}$ such that
$$\dist(i_n(X_n),i_m(X_m))\xrightarrow{m+n\to\infty}\infty.$$

Let us recall that a Banach space $(A,\lvert\lvert\cdot\rvert\rvert)$ is {\em uniformly convex} if, for every $0<\varepsilon\leq 2$, there exists $\delta>0$ so that, for any two vectors $x,y\in A$ with $\lvert\lvert x\rvert\rvert=\lvert\lvert y\rvert\rvert=1$, the condition $\lvert\lvert x- y\rvert\rvert\geq\varepsilon$ implies that $\lvert\lvert(x+y)/2\rvert\rvert\leq 1-\delta$. Hilbert spaces are, in particular, uniformly convex Banach spaces.
\begin{theorem}[\cite{Laf}]\label{thm:Lafforgue}
	There exists a sequence $\{X_k\}_{k\in\N}$ of finite metric spaces such that, if a metric space $X$ contains a coarse disjoint union of $\{X_k\}_{k\in\N}$, then $X$ cannot be coarsely embedded into any uniformly convex Banach space.
\end{theorem}

The goal of this section is to prove the following result.
\begin{theorem}\label{thm:EH_not_emb}
$\mathcal{EH}_1^{<\omega}$ cannot be coarsely embedded into any uniformly convex Banach space.
\end{theorem}

Hence, Theorem \ref{theo:theoB} 
immediately follows thanks to Theorem \ref{thm:biLipschitz} since a composite of coarse embeddings is still a coarse embedding.

To prove Theorem \ref{thm:EH_not_emb}, we intend to apply Theorem \ref{thm:Lafforgue}. Following the approach used in \cite{WeiYamZav}, let us first isometrically embed an arbitrary finite metric space into a more manageable space.
	
\begin{lemma}[Kuratowski embedding]\label{lemma:Kuratowski}
For every finite metric space $X$, there are $m,n\in\N\setminus\{0\}$ such that $X$ can be isometrically embedded into $([0,m]^n,d_m^n)$, where $d_m^n((x_i)_i,(y_i)_i)=\max_{i=1,\dots,n}\lvert x_i-y_i\rvert$ for every $(x_i)_i,(y_i)_i\in[0,m]^n$.
\end{lemma}

Therefore, in order to apply Theorem \ref{thm:Lafforgue}, we show that $\mathcal{EH}_1^{<\omega}$ 
contains a coarse disjoint union of the family $\{[0,m]^n\mid m,n\in\N\setminus\{0\}\}$. We intend to define the coarse embeddings $\varphi_m^n\colon[0,m]^n\to[\R]^{<\omega}$ recursively. To do it, 
let us fix a bijection $T\colon(\N\setminus\{0\})^2\to\N\setminus\{0\}$ defined as follows: for every pair $(m,n)\in\N$,
$$T(m,n)=m+\sum_{i=2}^{m+n-1}(i-1)=m+\frac{1}{2}(m+n-2)(m+n-1).$$
We represent in Figure \ref{fig:T} the sequence of the points $T(m,n)$.
\begin{figure}
	\centering
	\begin{tikzpicture}[>=latex]
		\draw[help lines] (1,1) grid (5,5);
        \draw (1,1) node[below left]{$1$};
        \draw (2,1) node[below]{$2$};
        \draw (3,1) node[below]{$3$};
        \draw (4,1) node[below]{$4$};
        \draw (1,2) node[left]{$2$};
        \draw (1,3) node[left]{$3$};
        \draw (1,4) node[left]{$4$};
		\draw[thick,->] (1,1)--(2,1);
		\draw[thick,->] (2,1)--(1,2);
		\draw[thick,->] (1,2)--(3,1);
		\draw[thick,->] (3,1)--(2,2);
		\draw[thick,->] (2,2)--(1,3);
		\draw[thick,->] (1,3)--(4,1);
		\draw[thick,->] (4,1)--(3,2);
		\draw[thick,->] (3,2)--(2,3);
		\draw[thick,->] (2,3)--(1,4);
		\draw[thick,->] (1,4)--(5,1);
		\draw[thick,->] (5,1)--(4,2);
		\draw[thick,->] (4,2)--(3,3);
		\draw[thick,->] (3,3)--(2,4);
		\draw[thick,->] (2,4)--(1,5);
	\end{tikzpicture}
	\caption{A representation of the map $T^{-1}\colon\N\setminus\{0\}\to(\N\setminus\{0\})^2$.}\label{fig:T}
\end{figure}
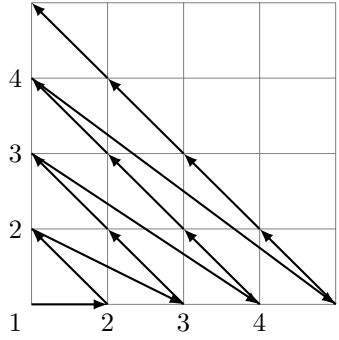
By construction, for every two pairs $(m,n),(m^\prime,n^\prime)\in\N\setminus\{0\}$, $m+n\leq m^\prime+n^\prime$ provided that $T(m,n)\leq T(m^\prime,n^\prime)$. For the sake of simplicity, for two pairs $(m,n),(m^\prime,n^\prime)\in(\N\setminus\{0\})^2$, let us denote $(m,n)\preceq(m^\prime,n^\prime)$ if $T(m,n)\leq T(m^\prime,n^\prime)$, and $(m,n)\prec(m^\prime,n^\prime)$ if $(m,n)\preceq(m^\prime,n^\prime)$ and $(m,n)\neq(m^\prime,n^\prime)$.

To construct the maps $\varphi_m^n$, we need various parameters. Let us define a sequence $\{a_i(m)\}_{i\in\N\setminus\{0\}}$ of positive real values depending on $m$ as follows: $$a_i(m)=4m(i-1).$$
Furthermore, for every $m,n\in\N\setminus\{0\}$, we inductively construct
$$D(m,n)=\max\bigg\{4m(n+2),\max_{(m^\prime,n^\prime)\prec (m,n)}D(m^\prime,n^\prime)+m+2^{T(m,n)}\bigg\}.$$
Thus, $D(1,1)=12$ (according to the notation that $\max\emptyset=-\infty$), and $D(m,n)\geq 4m(n+2)=a_n(m)+12m$. 

Let us define, for every $m,n\in\N\setminus\{0\}$, a map $\varphi_m^n\colon[0,m]^n\to [\R]^{=n+1}$ as follows: for every $(x_i)_i\in[0,m]^n$,
$$\varphi_m^n((x_i)_i)=\{a_i(m)+x_i\mid i=1,\dots,n\}\cup\{D(m,n)\}$$
(see Figure \ref{fig:varphi}). This construction should be compared with that provided in \cite{WeiYamZav}. The crucial difference is the last point $D(m,n)$. Its purpose is 
to disincentivise the action of isometries 
and conveniently increase the diameter of the image of $\varphi_m^n$ (see Lemma \ref{lemma:varphi}).
\begin{figure}
	\centering
\begin{tikzpicture}
	\draw[gray] (0,0)--(6,0) (7,0)--(10,0) (11,0)--(12,0);
	\draw[gray,dashed] (10,0)--(11,0) (6,0)--(7,0);
	\draw (0,-0.8)--(0,-1)--(1,-1)node[pos=0.5,below]{$m$} -- (1,-0.8);
    \draw (1,-0.5)--(1,-0.7)--(4,-0.7)node[pos=0.5,below]{$3m$}--(4,-0.5);
      \draw (9,-0.5)--(9,-0.7)--(12,-0.7)node[pos=0.5,below]{$\geq 11m$}--(12,-0.5);
	
	\draw (4,-0.8)--(4,-1)--(5,-1)node[pos=0.5,below]{$m$} -- (5,-0.8);
	\draw (8,-0.8)--(8,-1)--(9,-1)node[pos=0.5,below]{$m$} -- (9,-0.8);
	\draw[red] (0,0.6)--(0,0.8)--(0.6,0.8)node[pos=0.5,above]{$x_1$}--(0.6,0.2);
	\fill[red] (0.6,0) circle (2pt);
	\draw[red] (4,0.6)--(4,0.8)--(4.3,0.8)node[pos=0.5,above]{$x_2$}--(4.3,0.2);
	\fill[red] (4.3,0) circle (2pt);
	\draw[red] (8,0.6)--(8,0.8)--(8.8,0.8)node[pos=0.5,above]{$x_n$}--(8.8,0.2);
	\fill[red] (8.8,0) circle (2pt);
	\fill[blue](0.4,0) circle (2pt) (4.6,0) circle (2pt) (8.3,0) circle (2pt);
	\draw[blue] (0,-0.2)--(0,-0.4)--(0.4,-0.4)node[pos=0.5,below]{$y_1$}--(0.4,-0.2);
	\draw[blue] (4,-0.2)--(4,-0.4)--(4.6,-0.4)node[pos=0.5,below]{$y_2$}--(4.6,-0.2);
	\draw[blue] (8,-0.2)--(8,-0.4)--(8.3,-0.4)node[pos=0.5,below]{$y_n$}--(8.3,-0.2);
		\draw (0,-0.1)--(0,0.1);
	\draw (1,-0.1)--(1,0.1);
	\draw (4,-0.1)--(4,0.1);
	\draw (5,-0.1)--(5,0.1);
	\draw (8,-0.1)--(8,0.1);
	\draw (9,-0.1)--(9,0.1);
	\draw (12,-0.1)--(12,0.1);
 \fill (12,0) circle (2pt);
 \draw (12,0) node[above]{$D(m,n)$};
 \draw (0,0) node[above]{$a_1(m)$};
 	\draw (4,0) node[above]{$a_2(m)$};
	\draw (8,0) node[above]{$a_n(m)$};
\end{tikzpicture}

\caption{A representation of the images of two points $(x_i)_i,(y_i)_i\in[0,m]^n$ along $\varphi_m^n$. The subset $\varphi_m^n((x_i)_i)$ is given by the red dots and the black dot, while $\varphi_m^n((y_i)_i)$ consists of the blue dots and the black dot. In the picture, we can see that 
$d_H(\varphi_m^n((x_i)_i),\varphi_m^n((y_i)_i))=d_m^n((x_i)_i,(y_i)_i)$.}
\label{fig:varphi}
\end{figure}
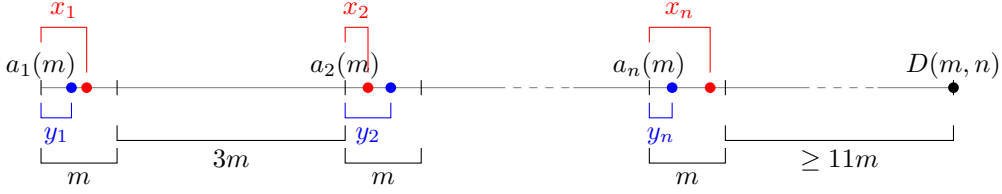

\begin{lemma}\label{lemma:varphi}
The map $\varphi_m^n\colon([0,m]^n,d_m^n)\to\mathcal{EH}_1^{<\omega}
$ is a coarse embedding whose control functions are independent of $m$ and $n$. More precisely, $\varphi_m^n$ is a $(\rho_-,\rho_+)$-coarse embedding, where $\rho_-\colon x\mapsto x/2$ and $\rho_+=id$. Furthermore, for every $(x_i)_i\in[0,m]^n$, $\diam\varphi_m^n((x_i)_i)\in[D(m,n)-m,D(m,n)]$.
\end{lemma}
\begin{proof}
It is easy to see that $$\EH(\varphi_m^n((x_i)_i),\varphi_m^n((y_i)_i))\leq d_H(\varphi_m^n((x_i)_i),\varphi_m^n((y_i)_i))= d_m^n((x_i)_i,(y_i)_i)$$
for every $(x_i)_i,(y_i)_i\in[0,m]^n$ (see also Figure \ref{fig:varphi}). We want to prove that $\rho_-$ is the other control function. Let $f\in\Iso(\R^d)$ be an isometry such that 
\begin{equation}\label{eq:varphi}
d_H(\varphi_m^n((x_i)_i),f(\varphi_m^n((y_i)_i)))\leq d_m^n((x_i)_i,(y_i)_i)\leq m.\end{equation}
The isometry $f$ is the composite of a translation $g$ and a rotation $h$. If $n=1$, then $\lvert\varphi_m^n((x_i)_i)\rvert=2$, and so, without loss of generality, $h$ can be taken as the identity. 
\begin{claim}\label{claim:no_rotation}
Assume that $n\geq 2$. Then, $h=id$.
\end{claim}
\begin{proof}
Suppose, by contradiction, that $h\neq id$. According to \eqref{eq:varphi}, and because of the definition of $\varphi_m^n$, 
\begin{equation}\label{eq:varphi_1}
\lvert f(D(m,n))-(a_1(m)+x_1)\rvert\leq m.
\end{equation}
Since $\lvert (a_1(m)+x_1)-(a_2(m)+x_1)\rvert\geq\lvert a_1(m)-a_2(m)\rvert-m>2m$, \eqref{eq:varphi_1} implies that $\lvert f(D(m,n))-(a_2(m)+x_2(m))\rvert>m$. Therefore, because of \eqref{eq:varphi}, 
\begin{equation}\label{eq:varphi_2}\lvert (a_2(m)+x_2)-f(a_n(m)+y_n)\rvert\leq m.\end{equation}
Using the fact that $f$ is an isometry, the triangular inequality, \eqref{eq:varphi_1} and \eqref{eq:varphi_2}, the following chain of inequalities descends:
$$\begin{aligned}
11m&\,\leq\lvert D(m,n)-a_n(m)\rvert-m\leq\lvert D(m,n)-(a_n(m)+y_n)\rvert=\lvert f(D(m,n))-f(a_n(m)+y_n)\rvert\leq\\
&\,\leq\lvert f(D(m,n))-(a_1(m)+x_1)\rvert+\lvert (a_1(m)+x_1)-(a_2(m)+x_2)\rvert+\lvert (a_2(m)+x_2)-f(a_n(m)+y_n)\rvert\leq\\
&\,\leq m+5m+m=7m.
\end{aligned}$$
Thus, we obtain a contradiction.
\end{proof}
We can now assume that $f=g$ is a translation. Using the triangular inequality, and since $a_{i+1}(m)-a_i(m)=4m$ for every $i\in\{1,\dots,n\}$, it can be easily check that,
$$\lvert (a_i(m)+x_i)-f(a_i(m)+y_i)\rvert\leq m,\text{ for every $i\in\{1,\dots,n\}$, and }\lvert D(m,n)-f(D(m,n))\rvert\leq m.$$

The point $D(m,n)$, common to both $\varphi_m^n((x_i)_i)$ and $\varphi_m^n((y_i)_i)$, misaligns as soon as $f$ is non-trivial. Therefore, $f$ creates a trade-off between the misalignment just described and a better alignment of the other $n$ pairs of points ($a_i(m)+x_i$ and $a_i(m)+y_i$, for every $i\in\{1,\dots,n\}$). Thus, the best Hausdorff distance that can be achieved is at most $d_H(\varphi_m^n((x_i)_i),\varphi_m^n((y_i)_i))/2=d_m^n((x_i)_i,(y_i)_i)/2$, and so the claim.

The final assertion trivially follows from the definition of $\varphi_m^n$. 
\end{proof}

The following result is immediate, but it is an important step in the proof of Theorem \ref{thm:EH_not_emb}.
\begin{lemma}\label{lemma:dist_via_diameter}
If $\mathcal Y$ and $\mathcal Z$ are two families of finite subsets of a metric space $X$, 
$$\dist_{d_H}(\mathcal Y,\mathcal Z)\geq\frac{\inf_{Y\in\mathcal Y\text{, }Z\in\mathcal Z}\lvert\diam Y-\diam Z\rvert}{2}.$$
Furthermore, if $X=\R^d$, then 
$$\dist_{\EH}(\mathcal Y,\mathcal Z)\geq\frac{\inf_{Y\in\mathcal Y\text{, }Z\in\mathcal Z}\lvert\diam Y-\diam Z\rvert}{2}.$$
\end{lemma}
\begin{proof}
Since, for every pair of subsets $Y$ and $Z$ of $X$, 
$$d_H(Y,Z)\geq\frac{\lvert\diam Y-\diam Z\rvert}{2},$$
the first inequality is immediate. The second one follows from the fact that an isometry's action does not change the diameter of a subset.
\end{proof}
As a consequence of Lemma \ref{lemma:dist_via_diameter}, if $\mathcal Y$ and $\mathcal Z$ are two families of finite subsets of $\R^d$ with the property that, for every $Y\in\mathcal Y$ and $Z\in\mathcal Z$, $\diam Y\geq\diam Z$, 
\begin{equation}\label{eq:dist_via_diameter}
\dist_{\EH}(\mathcal Y,\mathcal Z)\geq\frac{\inf_{Y\in\mathcal Y}\diam Y-\sup_{Z\in\mathcal Z}\diam Z}{2}.
\end{equation}

We now have all the tools to show our main result, Theorem \ref{thm:EH_not_emb}, and its consequence, Theorem \ref{theo:theoB}.

\begin{proof}[Proof of Theorem \ref{thm:EH_not_emb}]
We intend to use Lafforgue's result (Theorem \ref{thm:Lafforgue}), and, thanks to Lemma \ref{lemma:Kuratowski}, we need to show that $\mathcal{EH}_1^{<\omega}$ 
contains a coarse disjoint union of $\{([0,m]^n,d_m^n)\}_{m,n\in\N\setminus\{0\}}$. According to Lemma \ref{lemma:varphi}, $\{\varphi_m^n\colon[0,m]^n\to\mathcal{EH}_1^{<\omega}
\}$ is a coarse embedding. It remains to show that 
$$\dist_{\EH}(\varphi_m^n([0,m]^n),\varphi_{m^\prime}^{n^\prime}([0,m^\prime]^{n^\prime}))\xrightarrow{T(m,n)+T(m^\prime,n^\prime)\to\infty}\infty.$$ Without loss of generality, we assume that  $(m^\prime,n^\prime)\prec(m,n)$. Then, according to Lemma \ref{lemma:varphi}, \eqref{eq:dist_via_diameter}, and the definition of $D(m,n)$,
\begin{equation}\label{eq:coarse_disj_union}\dist_{\EH}(\varphi_m^n([0,m]^n),\varphi_{m^\prime}^{n^\prime}([0,m^\prime]^{n^\prime}))\geq\frac{D(m,n)-D(m^\prime,n^\prime)-m}{2}\geq 2^{T(m,n)-1}.\end{equation}
From \eqref{eq:coarse_disj_union}, the desired result descends.
\end{proof}


\section{Bi-Lipschitz non-embeddability into finite-dimensional spaces}\label{sec:non_bi_lip_emb}

Let us consider the isometry classes of all finite subsets of the real line with diameter at most $R$. This set can be identified with the isometry classes of elements in $[[0,R]]^{<\omega}$. Let us denote by $\EHs_{[0,R]}^{<\omega}$ this space equipped with the Euclidean-Hausdorff distance. 

Since, for every pair of metric spaces $X,Y\in\GHs_{\leq R}^{<\omega}$,
$$\GH(X,Y)\leq \frac{\max\{\diam X,\diam Y\}}{2}\leq\frac{R}{2},$$
the diameter of $\GHs_{\leq R}^{<\omega}$ is finite. Therefore, it can be trivially coarsely embedded into a one-point metric space. However, we can still provide non-embeddability results if we restrict the class of embeddings. 
More precisely, we prove the following.

\begin{theorem}\label{thm:EH_leq_R_non_embed}
$\EHs_{[0,R]}^{<\omega}$ cannot be bi-Lipschitz embedded into a finite-dimensional Hilbert space.
\end{theorem}
Immediately, we can deduce Theorem \ref{theo:theoC} 
since Theorem \ref{thm:biLipschitz} provides a bi-Lipschitz embedding.

To prove Theorem \ref{thm:EH_leq_R_non_embed}, we use the Assouad dimension. 
\begin{definition}
Given a metric space $(X,d)$, a subset $E\subseteq X$ and $r>0$, we denote by $N_r(E)$ the least number of open balls of radius less or equal to $r$ that cover $E$. Then, the {\em Assouad dimension of $X$} (\cite{Ass,Bou}) is the value
$$\dim_AX=\inf\{\alpha>0\mid\exists C>0:\forall r>0,\forall\beta\in(0,1],\sup_{x\in X}N_{\beta r}(B_d^o(x,r))<C\beta^{-\alpha}\},$$
where $B_d^o(x,r)$ denote the open ball centred in $x$ with radius $r$.
\end{definition}
This dimension notion was introduced precisely to prove obstructions to bi-Lipschitz embed metric spaces into an Euclidean space. In particular, the following properties lead to the desired conclusion (see, for example, \cite{Rob}):
\begin{compactenum}[$\bullet$]
\item if $\varphi\colon X\to Y$ is a bi-Lipschitz embedding between metric spaces, then $\dim_AX=\dim_Aim(\varphi)\leq\dim_AY$;
\item for every $n\in\N$, $\dim_A\R^n=n$.
\end{compactenum}
Therefore, once we prove that $\dim_A\EHs_{[0,R]}^{<\omega}=\infty$ (Proposition \ref{prop:dimA}), Theorem \ref{thm:EH_leq_R_non_embed} immediately follows.

Let us mention that the same strategy was used in \cite{CarBau} to prove that spaces of persistence diagrams cannot be bi-Lipschitz embedded into a finite-dimensional Hilbert space.

\begin{proposition}\label{prop:dimA}
$\dim_A\EHs_{[0,R]}^{<\omega}=\infty$.
\end{proposition}
\begin{proof}
We want to provide, for every $\alpha>0$ and every $C>0$, a radius $r>0$, a constant $\beta\in(0,1]$, a finite subset $A$ of $[0,R]$ and $M=\lceil C\beta^{-\alpha}+1\rceil$-many finite subsets $A_1,\dots,A_M$ of $[0,R]$ with the following properties: $\EH(A,A_i)<r$ and $\EH(A_i,A_j)>2\beta r$ for every $i,j\in\{1,\dots,M\}$.  Therefore, the open ball centred in $A$ with radius $r$ cannot be covered by fewer than $M$-many open balls with radius $\beta r$ ($A_i$ and $A_j$ are contained in the same ball only if $i=j$), which will conclude the proof since $M>C\beta^{-\alpha}$. 

Define $l=\frac{R}{2(M+1)}$, $r=l/6$, $s=2r/3$ and $\beta=1/2$. We construct the subsets $A$ and $A_i$ as follows:
\begin{gather*}A=\{j\cdot l\mid j\in\{0,\dots,M+1\}\}\cup\{R\},\quad\text{and}\\ 
A_i=\{j\cdot l+s\mid j\in\{1,\dots,M\}\setminus\{i\}\}\cup\{i\cdot l-s\}\cup\{0,R/2,R\}.\end{gather*}
The subset $A$ and $A_i$ are represented in Figure \ref{fig:dim_A}. It is clear that $\EH(A,A_i)<r$ since $s<r$ (actually, $d_H(A,A_i)=s<r$). It remains to show that, for every pair of distinct indices $i,j\in\{1,\dots,M\}$, $\EH(A_i,A_j)\geq 2s>r$. 

Let $i,j\in\{1,\dots,M\}$ be two distinct indices. Assume, by contradiction, that $\EH(A_i,A_j)<2s$, and let $f$ be an isometry of $\R$ such that $d_H(A_i,f(A_j))< 2s$. Adapting the argument of Claim \ref{claim:no_rotation}, we can assume that $f$ is a translation. 

For the sake of simplicity, for every $k\in\{1,\dots,n\}$, we name the points of $A_k$ as 
\begin{gather*}a_0^k=0, \quad a_1^k=l+s,\quad\dots\quad, a_{k-1}^k=(k-1)l+s,\quad a_{k}^k=kl-s,\\ a_{k+1}^k=(k+1)l+s,\quad\dots\quad, a_M^k=Ml+s,\quad a_{M+1}^k=R/2,\quad a_{M+2}^k=R. \end{gather*}
Again, following the strategy used in the proof of Lemma \ref{lemma:varphi}, we can show that, for every $k\in\{0,\dots,M+2\}$, 
\begin{equation}\label{eq:a_k_with_a_k}\lvert a_k^i-f(a_k^j)\rvert< 2s.\end{equation}
Assume, without loss of generality, that $i<j$. Then, in particular, 
\begin{equation}\label{eq:a_i_and_j}\lvert a_i^j-a_j^j\rvert=(j-i)-2s\quad\text{and}\quad\lvert a_i^i-a_j^i\rvert=(j-i)+2s.\end{equation} Using the triangular inequality, \eqref{eq:a_k_with_a_k}, \eqref{eq:a_i_and_j} and the fact that $f$ is an isometry we obtain the following chain of inequalities:
\begin{align*}
l(j-i)+2s=&\,\lvert a_i^i-a_j^i\rvert\leq \lvert a_i^i-f(a_i^j)\rvert+\lvert f(a_i^j)-f(a_j^j)\rvert+\lvert f(a_j^j)-a_j^i\rvert<\\
<&\,\lvert a_i^j-a_j^j\rvert+4s=(j-i)l+2s.
\end{align*}
Hence, a contradiction.
\end{proof}

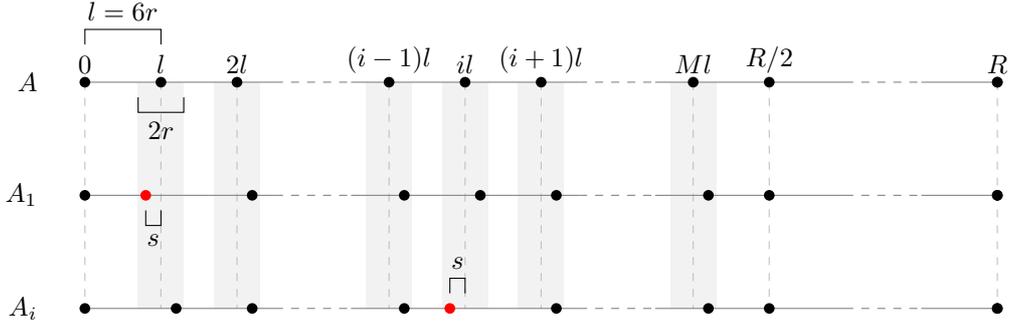
\begin{figure}
\centering
\begin{tikzpicture}
\fill[black!5!white] (0.7,0)--(1.3,0)--(1.3,-3)--(0.7,-3)--(0.7,0);
 \begin{scope}[shift={({1},{0})}]
 \fill[black!5!white] (0.7,0)--(1.3,0)--(1.3,-3)--(0.7,-3)--(0.7,0);
 \end{scope}
 \begin{scope}[shift={({3},{0})}]
 \fill[black!5!white] (0.7,0)--(1.3,0)--(1.3,-3)--(0.7,-3)--(0.7,0);
 \end{scope}
 \begin{scope}[shift={({4},{0})}]
 \fill[black!5!white] (0.7,0)--(1.3,0)--(1.3,-3)--(0.7,-3)--(0.7,0);
 \end{scope}
 \begin{scope}[shift={({5},{0})}]
 \fill[black!5!white] (0.7,0)--(1.3,0)--(1.3,-3)--(0.7,-3)--(0.7,0);
 \end{scope}
 \begin{scope}[shift={({7},{0})}]
 \fill[black!5!white] (0.7,0)--(1.3,0)--(1.3,-3)--(0.7,-3)--(0.7,0);
 \end{scope}
\draw[lightgray,dashed] (0,0)--(0,-3) (1,0)--(1,-3) (2,0)--(2,-3) (4,0)--(4,-3) (5,0)--(5,-3) (6,0)--(6,-3) (8,0)--(8,-3) (9,0)--(9,-3) (12,0)--(12,-3);
\draw (-0.5,0)node[left]{$A$};
\draw[gray] (0,0)--(2.5,0) (3.5,0)--(6.5,0) (7.5,0)--(10,0) (11,0)--(12,0);
\draw[gray,dashed] (2.5,0)--(3.5,0) (6.5,0)--(7.5,0) (10,0)--(11,0);
\fill (1,0) circle (2pt) (2,0)  circle (2pt) (4,0) circle (2pt) (5,0) circle (2pt) (6,0) circle (2pt) (8,0) circle (2pt) (12,0) circle (2pt);
\fill
(0,0) circle (2pt) (9,0) circle (2pt) (12,0) circle (2pt);
\draw (0.7,-0.2)--(0.7,-0.4)--(1.3,-0.4)node[pos=0.5,below]{$2r$}--(1.3,-0.2);
\draw (0,0.5)--(0,0.7)--(1,0.7)node[pos=0.5,above]{$l=6r$}--(1,0.5);

 \draw (12,0) node[above]{$R$};
 \draw (0,0) node[above]{$0$};
  \draw (1,0) node[above]{$l$};
 \draw (2,0) node[above]{$2l$};
 \draw (4,0) node[above]{$(i-1)l$};
 \draw (5,0) node[above]{$il$};
 \draw (6,0) node[above]{$(i+1)l$};
 \draw (8,0) node[above]{$Ml$};
 \draw (9,0) node[above]{$R/2$};

 \begin{scope}[shift={({0},{-1.5})}]
\draw (-0.5,0)node[left]{$A_1$};
\draw[gray] (0,0)--(2.5,0) (3.5,0)--(6.5,0) (7.5,0)--(10,0) (11,0)--(12,0);
\draw[gray,dashed] (2.5,0)--(3.5,0) (6.5,0)--(7.5,0) (10,0)--(11,0);
\fill
(2.2,0)  circle (2pt) (4.2,0) circle (2pt) (5.2,0) circle (2pt) (6.2,0) circle (2pt) (8.2,0) circle (2pt) (12,0) circle (2pt);
\fill[red] 
(0.8,0) circle (2pt);
\fill
(0,0) circle (2pt) (9,0) circle (2pt) (12,0) circle (2pt);
\draw (1,-0.2)--(1,-0.4)--(0.8,-0.4)node[pos=0.5,below]{$s$}--(0.8,-0.2);


\end{scope}

 \begin{scope}[shift={({0},{-3})}]
\draw (-0.5,0)node[left]{$A_i$};
\draw[gray] (0,0)--(2.5,0) (3.5,0)--(6.5,0) (7.5,0)--(10,0) (11,0)--(12,0);
\draw[gray,dashed] (2.5,0)--(3.5,0) (6.5,0)--(7.5,0) (10,0)--(11,0);
\fill
(1.2,0) circle (2pt) (2.2,0)  circle (2pt) (4.2,0) circle (2pt) 
(6.2,0) circle (2pt) (8.2,0) circle (2pt) (12,0) circle (2pt);
\fill[red] (4.8,0) circle (2pt);
\fill
(0,0) circle (2pt) (9,0) circle (2pt) (12,0) circle (2pt);
\draw (4.8,0.2)--(4.8,0.4)--(5,0.4)node[pos=0.5,above]{$s$}--(5,0.2);


\end{scope}
\end{tikzpicture}
\caption{A representation of the subsets $A$, $A_1$ and $A_i$ defined in the proof of Proposition \ref{prop:dimA}. The distinctive points $l-s\in A_1$ and $il-s\in A_i$ are emphasised in red. The light grey strips are meant to visualise the fact that $d_H(A,A_1)<r$ and $d_H(A,A_i)<r$.
}\label{fig:dim_A}
\end{figure}

\begin{remark}
   As a byproduct of the proof of Proposition \ref{prop:dimA}, we obtain that the Assouad dimension of the family of all finite subsets of an interval equipped with the Hausdorff distance is infinite. Indeed, in the notation of the mentioned proof, $d_H(A,A_i)=s<r$, but $d_H(A_i,A_j)=2s>r$ for every pair of distinct indices $i,j\in\{1,\dots,M\}$.
\end{remark}

\subsection{Questions about bi-Lipschitz embeddability and Assouad dimension}\label{sub:q}

Let us conclude the presentation with a discussion about potential future research directions concerning the bi-Lipschitz embeddability of the Gromov-Hausdorff space. 
First of all, Theorem \ref{theo:theoC} leaves the following question open.
\begin{question}\label{q:GH_finite_biL}
Can $\mathcal{GH}_{\leq R}^{<\omega}$ be bi-Lipschitz embedded into an infinite-dimensional Hilbert space?
\end{question}

Furthermore, it is natural to ask what the embeddability properties are if we bound the cardinality of the metric spaces as in Section \ref{sec:asdim}. Let us denote $\mathcal{GH}_{\leq R}^{\leq n}=\mathcal{GH}^{\leq n}\cap\mathcal{GH}_{\leq R}^{<\omega}$ for $n\in\N$ and $R>0$. 

\begin{question}\label{q:GH_bi_L}
Can $\mathcal{GH}^{\leq n}$ and $\mathcal{GH}^{\leq n}_{\leq R}$ be bi-Lipschitz embedded into a (finite-dimensional) Hilbert space?
\end{question}

To approach Question \ref{q:GH_bi_L}, we may investigate the Assouad dimension of those spaces. 
The inequalities $$\dim_A\mathcal{GH}^{\leq n}\geq \dim_A\mathcal{GH}^{\leq n}_{\leq R}\geq\frac{n(n-1)}{2}$$ can be derived similarly to the proof of Lemma \ref{lemma:asdim_geq}. Indeed, according to \cite[Theorem 4.1]{IliIvaTuz}, $\mathcal{GH}^{\leq n}_{\leq R}$ contains a subspace isometric to an open ball in $\R^{n(n-1)/2}$ of sufficiently small radius, which has Assouad dimension $n(n-1)/2$ (\cite[Lemma 9.6(iii)]{Rob}). Hence, monotonicity and bi-Lipschitz invariance of the dimension imply the desired estimate. However, to the best of the author's knowledge, no upper bounds are known for the dimension of those spaces, and the following questions remain open.
\begin{question}\label{q:dim_A_GH}
What are $\dim_A\mathcal{GH}^{\leq n}$ and $\dim_A\mathcal{GH}^{\leq n}_{\leq R}$?
\end{question}
\begin{question}\label{q:dim_A_GH_finite}
Are $\dim_A\mathcal{GH}^{\leq n}$ and $\dim_A\mathcal{GH}^{\leq n}_{\leq R}$ finite?
\end{question}


A connection between Questions \ref{q:dim_A_GH_finite} and \ref{q:GH_bi_L} has already been exploited to deduce Theorem \ref{theo:theoC}, namely, infinite Assouad dimension prevents the existence of bi-Lipschitz embeddings into finite-dimensional Hilbert spaces. 
However, unlike the situation described for asymptotic dimension and coarse embeddings, a positive answer to Question \ref{q:dim_A_GH_finite} does not imply the existence of a bi-Lipschitz embedding into some $\R^n$. Indeed, having finite Assouad dimension is not a sufficient condition for the existence of a bi-Lipschitz embedding even into some infinite-dimensional Hilbert space (\cite{Pan,Sem,LanPla,Laa}). However, some positive results can be proved at the cost of modifying the original metric space. 
For a metric space $(X,d)$ and $0<\varepsilon<1$, the {\em $\varepsilon$-snowflaking of $X$} is the metric space $(X,d^\varepsilon)$, where $d^\varepsilon(x,y)=(d(x,y))^\varepsilon$. 
\begin{theorem}[Assouad embedding theorem, \cite{Ass}]\label{thm:Ass_emb}
If $(X,d)$ is a metric space with finite Assouad dimension, then, for every $0<\varepsilon<1$, there exists a bi-Lipschitz embedding of $(X,d^\varepsilon)$ into $\R^n$ for some $n$ depending only on $\dim_AX$ and $\varepsilon$.
\end{theorem}
Let us also mention that, if $1/2<\varepsilon<1$, the parameter $n$ in Assouad embedding theorem can be chosen independently from $\varepsilon$ (\cite{NaoNei}, see also \cite{DavSni} for an explicit map construction). We address the interested reader to the monographs \cite{Rob,Fra} for more details and references. A positive answer to Question \ref{q:dim_A_GH_finite} could then motivate the search for computable bi-Lipschitz embeddings of the $\varepsilon$-snowflaking of $\mathcal{GH}^{\leq n}$ or $\mathcal{GH}^{\leq n}_{\leq R}$ into some finite-dimensional Hilbert space, further tightening the connection between computational topology and dimension theory.

\Addresses
 
\end{document}